\documentclass[12pt,twoside]{amsart}
\usepackage{amsmath, amsthm, amscd, amsfonts, amssymb, graphicx, color}
\usepackage{enumerate}
\usepackage[colorlinks=true,
linkcolor=magenta,
urlcolor=cyan,
citecolor=cyan]{hyperref}

\newtheorem{theorem}{Theorem}[section]
\newtheorem{lemma}{Lemma}[section]
\newtheorem{remark}{Remark}[section]

\newtheorem{corollary}{Corollary}[section]
\newtheorem{example}{Example}[section]
\newtheorem{proposition}{Proposition}[section]
\numberwithin{equation}{section}

\begin{document}
\title{Choi-Davis-Jensen's inequality without convexity}

\author{Jadranka Mi\'{c}i\'{c}}
\address{Jadranka Mi\'{c}i\'{c}, Faculty of Mechanical Engineering and Naval Architecture, University of Zagreb,
Ivana Lu\v ci\' ca 5, 10000 Zagreb, Croatia.}\email{jmicic@fsb.hr}

\author{Hamid Reza Moradi}
\address{Hamid Reza Moradi,  Young Researchers and Elite Club, Mashhad Branch, Islamic Azad University, Mashhad, Iran.}\email{hrmoradi@mshdiau.ac.ir}

\author{Shigeru Furuichi}
\address{Shigeru Furuichi,  Department of Information Science, College of Humanities and Sciences, Nihon University, 3-25-40, Sakurajyousui, Setagaya-ku, Tokyo, 156-8550, Japan.}\email{furuichi@chs.nihon-u.ac.jp}

 \maketitle
\begin{abstract}
We give the Choi-Davis-Jensen type inequality without using convexity.
Applying our main results, we also give new inequalities improving previous known results. In particular, we show some inequalities for relative operator entropies and quantum mechanical entropies.
\end{abstract}

\maketitle

\pagestyle{myheadings}
\markboth{\centerline {Choi-Davis-Jensen's inequality without convexity}}
{\centerline {J. Mi\'{c}i\'{c}, H.R. Moradi \& S. Furuichi}} \bigskip \bigskip

\section{\bf Introduction} \label{section1}
\vskip0.4 true cm
We assume that the reader is familiar with basic notions about operator theory.

Davis \cite{5} and Choi \cite{6} showed that if $\Phi :\mathcal{B}(\mathcal{H}) \rightarrow \mathcal{B}(\mathcal{K})$ is a unital positive linear map  and if $f$ is an operator convex function on an interval $I$, then the so-called {\it Choi-Davis-Jensen's inequality} (in short C-D-J inequality)
\begin{equation}\label{9}
f\left( \Phi \left( A \right) \right)\le \Phi \left( f\left( A \right) \right)
\end{equation}
holds for every self-adjoint operator $A$ on $\mathcal{H}$ whose spectrum is contained in $I$.

The inequality \eqref{9} can break down when the operator convexity is dropped. For instance, taking
\begin{equation*}
A=\left( \begin{matrix}
   4 & 1 & -1  \\
   1 & 2 & 1  \\
   -1 & 1 & 2  \\
\end{matrix} \right), \quad \Phi \left({{\left( {{a}_{ij}} \right)}_{1\le i,j\le 3}} \right)={{\left( {{a}_{ij}} \right)}_{1\le i,j\le 2}} \quad \mathrm{and}  \quad f\left( t \right)\equiv {{t}^{4}}.
\end{equation*}
By a simple computation, we have

\[\left( \begin{matrix}
   325 & 132  \\
   132 & 61  \\
\end{matrix} \right)={{\Phi }^{4}}\left( A \right)\nless \Phi \left( {{A}^{4}} \right)=\left( \begin{matrix}
   374 & 105  \\
   105 & 70  \\
\end{matrix} \right).\]
This example shows that the inequality \eqref{9} will be false if we replace the operator convex function by a general convex function. In \cite[Theorem 1]{1}, Mi\'ci\'c et al. pointed out that the inequality \eqref{9} holds true for real valued continuous convex functions with conditions on the bounds of the operators.

\medskip

The purpose of this paper is to obtain the C-D-J inequality for non-convex functions. Applying our main results, we give new inequalities improving previous known results such as the Kantorovich inequality, and bounds for relative operator entropies and quantum mechanical entropies.

\section{\bf Main results} \label{section2}
\vskip0.4 true cm
In this section we give our main results.
\begin{theorem}\label{th1}
 Let $f:I\to \mathbb{R}$ be continuous twice differentiable function such that $\alpha \le f''\le \beta $ where $\alpha ,\beta \in \mathbb{R}$ and let $\Phi :\mathcal{B}(\mathcal{H}) \rightarrow \mathcal{B}(\mathcal{K})$ be unital positive linear map. Then
\begin{equation}\label{th1-1}
f\left( \Phi \left( A \right) \right)\le \Phi \left( f\left( A \right) \right)+\frac{\beta -\alpha }{2}\left\{ \left( M+m \right)\Phi \left( A \right)-Mm \right\}+\frac{1}{2}\left( \alpha \Phi {{\left( A \right)}^{2}}-\beta \Phi \left( {{A}^{2}} \right) \right),
\end{equation}
and
\begin{equation}\label{th1-2}
\Phi \left( f\left( A \right) \right)\le f\left( \Phi \left( A \right) \right)+\frac{\beta -\alpha }{2}\left\{ \left( M+m \right)\Phi \left( A \right)-Mm \right\}+\frac{1}{2}\left( \alpha \Phi \left( {{A}^{2}} \right)-\beta \Phi {{\left( A \right)}^{2}} \right),
\end{equation}
for any self-adjoint operator $A$ on $\mathcal{H}$ with the spectrum $Sp\left( A \right)\subseteq \left[ m,M \right]\subset I$.
\end{theorem}

\medskip

In order to prove Theorem~\ref{th1}, we need the following lemma.

\begin{lemma}\label{lemma1}
 Let $f:I\to \mathbb{R}$ be continuous twice differentiable function such that $\alpha \le f''\le \beta $ on $I$, where $\alpha ,\beta \in \mathbb{R}$, and let $\Phi :\mathcal{B}(\mathcal{H}) \rightarrow \mathcal{B}(\mathcal{K})$ be unital  positive linear map. If $A$ is a self-adjoint operator on $\mathcal{H}$  with $Sp\left( A \right)\subseteq \left[ m,M \right]\subset I$ for some $m<M$, then
 \begin{equation}\label{lemma1-eq1}
\Phi \left( f\left( A \right) \right)\le L\left( \Phi \left( A \right) \right)-\frac{\alpha }{2}\left\{ \left( M+m \right)\Phi \left( A \right)-Mm-\Phi \left( {{A}^{2}} \right) \right\},
\end{equation}
\begin{equation}\label{lemma1-eq2}
\Phi \left( f\left( A \right) \right)\ge L\left( \Phi \left( A \right) \right)-\frac{\beta }{2}\left\{ \left( M+m \right)\Phi \left( A \right)-Mm-\Phi \left( {{A}^{2}} \right) \right\},
\end{equation}
\begin{equation}\label{lemma1-eq3}
f\left( \Phi \left( A \right) \right)\le L\left( \Phi \left( A \right) \right)-\frac{\alpha }{2}\left\{ \left( M+m \right)\Phi \left( A \right)-Mm-\Phi {{\left( A \right)}^{2}} \right\},
\end{equation}
 \begin{equation}\label{lemma1-eq4}
f\left( \Phi \left( A \right) \right)\ge L\left( \Phi \left( A \right) \right)-\frac{\beta }{2}\left\{ \left( M+m \right)\Phi \left( A \right)-Mm-\Phi {{\left( A \right)}^{2}} \right\},
 \end{equation}
 where
 \begin{equation}\label{lineL}
L\left( t \right):=\frac{M-t}{M-m}f\left( m \right)+\frac{t-m}{M-m}f\left( M \right),
 \end{equation}
 is the line that passes through the points $\left( m,f\left( m \right) \right)$ and $\left( M,f\left( M \right) \right)$.
\end{lemma}

\begin{proof}
Since $\alpha \le f''\le \beta $, then the function ${{g}_{\alpha }}\left( x \right):= f\left( x \right)-\frac{\alpha }{2}{{x}^{2}}$ is convex. So,
$$ g_{\alpha}((1-\lambda) a+ \lambda b) \leq (1-\lambda) g_{\alpha} (a)+ \lambda g_{\alpha}(b),$$ holds for any $0\leq \lambda \leq 1$ and $a,b \in I$.
It follows that
$$  f((1-\lambda) a+ \lambda b) \leq (1-\lambda)  f(a)  + \lambda  f(b) - \frac{\alpha }{2}~\lambda (1-\lambda) (a-b)^{2}.$$
Since any $t \in [m,M]$ can be written in the form $t= \frac{M-t}{M-m} m +\frac{t-m}{M-m}M$, and putting $\lambda = \frac{t-m}{M-m}$, $a=m$ and $b=M$ in the above inequality we have
 \begin{equation}\label{lemma1-5}
f\left( t \right)\le L\left( t \right)-\frac{\alpha }{2}\left( (M+m) t -mM -t^2 \right).
\end{equation}
Now, by using  the standard calculus of a self-adjoint operator $A$ to \eqref{lemma1-5} and next applying an unital  positive linear map $\Phi$ we obtain
$$\Phi(f(A)) \leq \Phi(L(A))-\frac{\alpha}{2} \left\{ (M+m) \Phi(A)-Mm-\Phi(A^2)\right\},$$
which gives the desired inequality \eqref{lemma1-eq1}.

By applying \eqref{lemma1-5} on $\Phi(A)$  we obtain \eqref{lemma1-eq3}.
Using the same technique as above for a convex function ${{g}_{\beta }}\left( t \right):= \frac{\beta }{2}{{x}^{2}}-f\left( x \right)$ we obtain
 \begin{equation*}
L\left( t \right)-\frac{\beta }{2}\left( (M+m) t -mM -t^2 \right) \le f\left( t \right),
\end{equation*}
which gives \eqref{lemma1-eq2} and \eqref{lemma1-eq4}.
\end{proof}

\medskip
From Lemma \ref{lemma1} we can derive Theorem \ref{th1}:

\begin{proof}[Proof of Theorem~\ref{th1}] Let $m <M$.  We obtain \eqref{th1-1} after combining \eqref{lemma1-eq3} with \eqref{lemma1-eq2} and we obtain \eqref{th1-2} after combining \eqref{lemma1-eq1} with \eqref{lemma1-eq4}.
\end{proof}

\begin{remark}
The inequality \eqref{th1-2} is a converse of  C-D-J inequality $f(\Phi(A)) \leq \Phi(f(A))$ for a non-convex function. The second term in  \eqref{th1-2} is always non-negative, while the sing of  third term in \eqref{th1-2} is not determined.
\end{remark}

\begin{example}
To illustrate Theorem \ref{th1} works properly, let $\Phi \left( A \right)=\left\langle Ax,x \right\rangle $, where $x=\left( \begin{matrix}
   \frac{1}{\sqrt{3}}  \\
   \frac{1}{\sqrt{3}}  \\
   \frac{1}{\sqrt{3}}  \\
\end{matrix} \right)$,
   $A=\left( \begin{matrix}
      1 & 0 & -1  \\
      0 & 3 & 1  \\
      -1 & 1 & 2  \\
   \end{matrix} \right)$
and $f\left( t \right)={{t}^{3}}$.
Of course we can choose $m=0.25$ and $M=3.8$. So after some calculations we see that
\[\begin{aligned}
   8&=f\left( \Phi \left( A \right) \right) \\
 & \lneqq \Phi \left( f\left( A \right) \right)+\frac{\beta -\alpha }{2}\left\{ \left( M+m \right)\Phi \left( A \right)-Mm \right\}+\frac{1}{2}\left( \alpha \Phi {{\left( A \right)}^{2}}-\beta \Phi \left( {{A}^{2}} \right) \right)\simeq 27.14, \\
\end{aligned}\]
and
\[\begin{aligned}
   24&=\Phi \left( f\left( A \right) \right) \\
 & \lneqq f\left( \Phi \left( A \right) \right)+\frac{\beta -\alpha }{2}\left\{ \left( M+m \right)\Phi \left( A \right)-Mm \right\}+\frac{1}{2}\left( \alpha \Phi \left( {{A}^{2}} \right)-\beta \Phi {{\left( A \right)}^{2}} \right)\simeq 43.54. \\
\end{aligned}\]
\end{example}
\begin{theorem}\label{th2}
Let $A$ be a self-adjoint operator with $Sp\left( A \right)\subseteq \left[ m,M \right]\subset I$ for some $m<M$. If $f:[m,M]\to (0,\infty)$ is a continuous twice differentiable function such that $\alpha \le f''$ on $[m,M]$, where $\alpha \in \mathbb{R}$, and if $\Phi :\mathcal{B}\left( \mathcal{H} \right)\to \mathcal{B}\left( \mathcal{K} \right)$ is unital  positive linear map, then
\begin{equation}\label{th2-1}
\begin{aligned}
  & \frac{1}{K\left( m,M,f \right)}\left\{ \Phi \left( f\left( A \right) \right)+\frac{\alpha }{2}\left[ \left( M+m \right)\Phi \left( A \right)-Mm-\Phi \left( {{A}^{2}} \right) \right] \right\} \\
 & \le f\left( \Phi \left( A \right) \right) \\
 & \le K\left( m,M,f \right)\Phi \left( f\left( A \right) \right)-\frac{\alpha }{2}\left[ \left( M+m \right)\Phi \left( A \right)-Mm-\Phi {{\left( A \right)}^{2}} \right], \\
\end{aligned}
\end{equation}
where
\begin{equation}\label{const-K}
  K(m,M,f)=\max \left\{\frac{1}{f(t)} \left( \frac{M-t}{M-m} f(m) + \frac{t-m}{M-m}f (M) \right) :  t \in[m,M] \right\}.
\end{equation}
\end{theorem}
\begin{proof}
By using \eqref{lemma1-eq3}, we have (see \cite[Corollary~4.12]{mpst2000})
\[\begin{aligned}
   f\left( \Phi \left( A \right) \right)&\le L\left( \Phi \left( A \right) \right)-\frac{\alpha }{2}\left[ \left( M+m \right)\Phi \left( A \right)-Mm-\Phi {{\left( A \right)}^{2}} \right] \\
 & \le K\left( m,M,f \right)\Phi \left( f\left( A \right) \right)-\frac{\alpha }{2}\left[ \left( M+m \right)\Phi \left( A \right)-Mm-\Phi {{\left( A \right)}^{2}} \right], \\
\end{aligned}\]
which gives RHS inequality of \eqref{th2-1}.
Also, by using \eqref{lemma1-eq1} and given that $0<m \leq \Phi(A)<M$, we obtain
$$
 \Phi(f(A))
   \leq K(m,M,f) f(\Phi(A)) -\frac{\alpha}{2} \left[ (M+m) \Phi(A)-Mm-\Phi(A^2)\right].
$$
 Since $K(m,M,f)> 0$ it follows
\begin{equation*}\label{th2-3}
f(\Phi(A)) \geq \frac{1}{K(m,M,f)} \left\{\Phi(f(A))+\frac{\alpha}{2} \left[ (M+m) \Phi(A)-Mm-\Phi(A^2)\right] \right\},
\end{equation*}
which is LHS inequality of \eqref{th2-1}.
\end{proof}

\begin{remark}
Let $A$ and $\Phi$ be as in Theorem~\ref{th2}. If $f:[m,M]\to (0,\infty)$ is a continuous twice differentiable function such that $f'' \le \beta$, where $\beta \in \mathbb{R}$, then by using \eqref{lemma1-eq2} and \eqref{lemma1-eq4}, we can obtain the following result
\begin{equation*}
\begin{aligned}
  & k\left( m,M,f \right) \Phi \left( f\left( A \right) \right)-\frac{\beta }{2}\left[ \left( M+m \right)\Phi \left( A \right)-Mm-\Phi {{\left( A \right)}^{2}} \right] \\
 & \le f\left( \Phi \left( A \right) \right) \\
 & \le \frac{1}{k\left( m,M,f \right)}  \left\{ \Phi \left( f\left( A \right) \right)+\frac{\beta }{2}\left[ \left( M+m \right)\Phi \left( A \right)-Mm-\Phi \left( {{A}^{2}} \right) \right] \right\}, \\
\end{aligned}
\end{equation*}
where
$ k(m,M,f)=\min \left\{\frac{L(t)}{f(t)} :  t \in[m,M] \right\}$ and
$L(t)$ is defined by \eqref{lineL}.
\end{remark}

\medskip

In the next corollary  we give a  refinement of  converse of  C-D-J inequality (see e.g.\ \cite{mpst2000}) $$\frac{1}{K(m,M,f)} \Phi(f(A)) \leq f(\Phi(A)) \leq K(m,M,f)  \Phi(f(A)),$$ for every strictly convex function $f$  on $[m,M]$, where
$K(m,M,f)>1$ is defined by \eqref{const-K}.
\begin{corollary} \label{cor1}
Let the assumptions of Theorem~\ref{th2} hold and $f$ be strictly  convex on $[m,M]$. Then
\begin{equation*}
\begin{aligned}
  & \frac{1}{K\left( m,M,f \right)}\Phi \left( f\left( A \right) \right) \\
 & \le \frac{1}{K\left( m,M,f \right)}\left\{ \Phi \left( f\left( A \right) \right)+\frac{\alpha }{2}\left[ \left( M+m \right)\Phi \left( A \right)-Mm-\Phi \left( {{A}^{2}} \right) \right] \right\} \\
 & \le f\left( \Phi \left( A \right) \right) \\
 & \le K\left( m,M,f \right)\Phi \left( f\left( A \right) \right)-\frac{\alpha }{2}\left[ \left( M+m \right)\Phi \left( A \right)-Mm-\Phi {{\left( A \right)}^{2}} \right] \\
 & \le K\left( m,M,f \right)\Phi \left( f\left( A \right) \right), \\
\end{aligned}
\end{equation*}
where $K(m,M,f)>1$ is defined by \eqref{const-K}.
\end{corollary}

\begin{proof}
Since $f$ is  strictly convex, then $0<\alpha \le f''$ on $[m,M]$. Given that $(M+m) \Phi(A)-Mm-\Phi(A^2) \geq 0$ is valid, the desired result follows by applying Theorem~\ref{th2}.
\end{proof}



\section{\bf Some applications}\label{section3}
\vskip0.4 true cm
The {\it non-commutative perspective} of the continuous function $f$ is defined by
\[{{\mathcal{P}}_{f}}\left( A|B \right):={{A}^{\frac{1}{2}}}f\left( {{A}^{-\frac{1}{2}}}B{{A}^{-\frac{1}{2}}} \right){{A}^{\frac{1}{2}}},\]
for every self-adjoint operator $B$ and strictly positive operator $A$ on a Hilbert space $\mathcal{H}$ (see \cite{Eba}). This notion is a generalization of the notion of the {\it commutative perspective} considered by Effros \cite{E}. Similar studies have been done in \cite{FFS}.

\medskip

The following result providing upper and lower bounds for the non-commutative perspective holds (for similar result see \cite[Theorem 1]{dragomir}). We use the notation $A{{\natural}_{p}}B$ to mean that ${{A}^{\frac{1}{2}}}{{\left( {{A}^{-\frac{1}{2}}}B{{A}^{-\frac{1}{2}}} \right)}^{p}}{{A}^{\frac{1}{2}}}$, $p\in \mathbb{R}$.
\begin{proposition} \label{prop3_1}
Let $A\in \mathcal{B}\left( \mathcal{H} \right)$ be strictly a positive operator and let $B \in \mathcal{B}\left( \mathcal{H} \right)$ be a positive operator which satisfies in the sandwich condition $mA\le B\le MA$ where $0<m<M$. If $f$ is twice continuously differentiable such that $\alpha \le f''\le \beta $, then
\[\begin{aligned}
   \frac{\beta }{2}\left( A{{\natural}_{2}}B+MmA-\left( M+m \right)B \right)&\le {{\mathcal{P}}_{f}}\left( A|B \right)-{{L}_{f}}\left( A|B \right) \\
 & \le \frac{\alpha }{2}\left( A{{\natural}_{2}}B+MmA-\left( M+m \right)B \right),
\end{aligned}\]
where
\[{{L}_{f}}\left( A|B \right) = \frac{1}{M-m}\left\{ \left( B-mA \right)f\left( M \right)+\left( MA-B \right)f\left( m \right) \right\}.\]
\end{proposition}
\begin{proof}
Since $\alpha \le f''\le \beta $, so two functions ${{g}_{\alpha }}\left( x \right)\equiv f\left( x \right)-\frac{\alpha }{2}{{x}^{2}}$ and ${{g}_{\beta }}\left( x \right)\equiv \frac{\beta }{2}{{x}^{2}}-f\left( x \right)$ are convex. Hence Jensen's inequality works for ${{g}_{\alpha }}\left( x \right)$ and ${{g}_{\beta }}\left( x \right)$. After substitution one obtains for any $t\in \left[ m,M \right]$,
$$
L\left( t \right)-\frac{\beta }{2}\left( t-m \right)\left( M-t \right)\le f\left( t \right),
$$
and
$$
f\left( t \right)\le L\left( t \right)-\frac{\alpha }{2}\left( t-m \right)\left( M-t \right),
$$
where
	\[L\left( t \right)\equiv \frac{1}{M-m}\left\{ \left( t-m \right)f\left( M \right)+\left( M-t \right)f\left( m \right) \right\}.\]
Thus we have
$$
\frac{\alpha}{2}(t-m)(M-t) \leq L(t) -f(t) \leq \frac{\beta}{2}(t-m)(M-t).
$$
Replacing $t$ with the positive operator ${{A}^{-\frac{1}{2}}}B{{A}^{-\frac{1}{2}}}$ and then multiplying both sides by ${{A}^{\frac{1}{2}}}$ we deduce the desired result.
\end{proof}

\medskip

The following is a simple consequence of Proposition \ref{prop3_1}.
\begin{remark}\label{remark3_1}
If we take $f_p(t) = \frac{1-t^p}{p}$ for $-1 \leq p \leq 1$ with $p \ne 0$ in Proposition \ref{prop3_1}, then we have the bounds of Tsallis relative operator entropy \cite{YKF2005} defined by

\[{{T}_{p}}\left( A|B \right) := \frac{A{{\natural}_{p}}B-A}{p},\quad \text{ for }-1\le p\le 1\text{ and }p\ne 0,\]
as follows.
Since $f_p''(t) =(1-p)t^{p-2} \geq 0$ for $-1 \leq p \leq 1$, we have $(1-p)M^{p-2}\leq f_p''(t) \leq (1-p)m^{p-2}$ for $0< m \leq t \leq M$ so that we can take $\alpha = (1-p)M^{p-2}$ and $\beta = (1-p)m^{p-2}$ in Proposition \ref{prop3_1}.
In addition, we obtain the following inequaliteis by the use of Proposition \ref{prop3_1} with simple calculations
\begin{eqnarray*}
&&-\frac{\alpha}{2} \left(A\natural_2 B +MmA-(M+m)B \right) \\
&& \leq T_p(A|B) + \frac{1}{p(M-m)}\left\{ (B-mA) (1-M^p) +(MA-B)(1-m^p)\right\} \\
&&-\frac{\beta}{2} \left(A\natural_2 B +MmA-(M+m)B \right).
\end{eqnarray*}
Thus we have the inequalities  with simple calculations
\[{\widetilde L_{m,M,p}}\left( {A\left| B \right.} \right) - \frac{{\left( {1 - p} \right)}}{{2{M^{2 - p}}}}{K_{m,M}}\left( {A,B} \right) \le {T_p}\left( {A\left| B \right.} \right) \le {\widetilde L_{m,M,p}}\left( {A\left| B \right.} \right) - \frac{{\left( {1 - p} \right)}}{{2{m^{2 - p}}}}{K_{m,M}}\left( {A,B} \right),\]
where
\[\left\{ \begin{array}{l}
{\widetilde L_{m,M,p}}\left( {A\left| B \right.} \right) = \frac{-1}{{p\left( {M - m} \right)}}\left\{ {\left( {M - m + Mm\left( {{M^{p - 1}} - {m^{p - 1}}} \right)} \right)A - \left( {{M^p} - {m^p}} \right)B} \right\},\\
{K_{m,M}}\left( {A,B} \right) = A{\natural _2}B + MmA - \left( {M + m} \right)B.
\end{array} \right.\]
Taking the limit $p \to 0$, we have the bounds of relative operator entropy \cite{FK1989} defined by
\[S\left( A|B \right):={{A}^{\frac{1}{2}}}\log \left( {{A}^{-\frac{1}{2}}}B{{A}^{-\frac{1}{2}}} \right){{A}^{\frac{1}{2}}},\]
as
\[{{\widetilde{L}}_{m,M}}\left( A|B \right)-\frac{1}{2{{M}^{2}}}{{K}_{m,M}}\left( A,B \right)\le S\left( A|B \right)\le {{\widetilde{L}}_{m,M}}\left( A|B \right)-\frac{1}{2{{m}^{2}}}{{K}_{m,M}}\left( A,B \right),\]
where
\[{\widetilde L_{m,M}}\left( {A\left| B \right.} \right) = \frac{1}{{\left( {M - m} \right)}}\left\{ {\left( {B - mA} \right)\log M + \left( {MA - B} \right)\log m} \right\},\]
due to $\lim_{p \to 0}T_p(A|B) = S(A|B)$ and $\lim_{p\to 0} f_p(t)=-\log t.$
\end{remark}

\medskip

\begin{proposition} \label{prop_3_2}
Let $A,B\in \mathcal{B}\left( \mathcal{H} \right)$ be two positive operators which satisfy the sandwich condition $mA\le B\le MA$ where $m<M$  and $\Phi :\mathcal{B}(\mathcal{H}) \rightarrow \mathcal{B}(\mathcal{K})$ be unital positive linear map on $\mathcal{B}\left( \mathcal{H} \right)$. If $f$ is twice continuously differentiable such that $\alpha \le f''\le \beta $, then
\[\begin{aligned}
  & \frac{\alpha -\beta }{2}\left\{ \left( M+m \right)\Phi \left( B \right)-Mm\Phi \left( A \right) \right\}+\frac{1}{2}\left( \beta \left( \Phi \left( A \right){{\natural}_{2}}\Phi \left( B \right) \right)-\alpha \Phi \left( A{{\natural}_{2}}B \right) \right) \\
 & \le {{\mathcal{P}}_{f}}\left( \Phi \left( A \right)|\Phi \left( B \right) \right)-\Phi \left( {{\mathcal{P}}_{f}}\left( A|B \right) \right) \\
 & \le \frac{\beta -\alpha }{2}\left\{ \left( M+m \right)\Phi \left( B \right)-Mm\Phi \left( A \right) \right\}+\frac{1}{2}\left( \alpha \left( \Phi \left( A \right){{\natural}_{2}}\Phi \left( B \right) \right)-\beta \Phi \left( A{{\natural}_{2}}B \right) \right). \\
\end{aligned}\]
\end{proposition}

\begin{proof}
In the inequalities of Proposition \ref{prop3_1}, we replace $\Phi(A)$, $\Phi(B)$ with $A, B$, then we have
\begin{eqnarray*}
&& \frac{\beta}{2}\left(\Phi(A) \natural_2\Phi(B) +Mm\Phi(A)-(M+m)\Phi(B) \right) \\
&& \leq \mathcal{P}_f (\Phi(A)|\Phi(B)) -L_f(\Phi(A)|\Phi(B) \\
&& \leq \frac{\alpha}{2}\left(\Phi(A) \natural_2\Phi(B) +Mm\Phi(A)-(M+m)\Phi(B) \right).
\end{eqnarray*}
Taking $\Phi$ in the both sides of the inequalities of Proposition \ref{prop3_1}, we have
\begin{eqnarray*}
&& \frac{\beta}{2}\left(\Phi(A\natural_2B) +Mm\Phi(A)-(M+m)\Phi(B)\right) \\
&& \leq \Phi\left( \mathcal{P}_f (A|B)\right)-\Phi(L_f(A|B)) \\
&& \leq \frac{\alpha}{2}\left(\Phi(A\natural_2B) +Mm\Phi(A)-(M+m)\Phi(B)\right).
\end{eqnarray*}
Thus we have the desired result since $L_f(\Phi(A)|\Phi(B)=\Phi(L_f(A|B))$.

\end{proof}

\medskip

Let $\rho$ be strictly positive operator with unit trace. (Such an operator is often called a {\it density operator} in quantum physics \cite{OP2004}.)
Then {\it von Neumann entropy} (quantum mechanical entropy) \cite[Section 4.3]{Bhatia2007},  \cite{OP2004} is defined by
$S(\rho) := -Tr[\rho \log \rho]$. In addition, the {\it quantum Tsallis entropy} \cite{FYK2004,Furu2008} is defined by $S_p(\rho) := \frac{Tr[\rho^{1-p} -\rho]}{p}$ for $-1\leq p \leq 1$ with $p\neq 0$. They have a non-negativity $S(\rho) \geq 0$ and $S_p(\rho) \geq 0$.
The Tsallis relative entropy \cite{FYK2004} in quantum mechanical system is also defined by
$$
D_p(\rho |\sigma) := \frac{Tr[\rho-\rho^{1-p}\sigma^p]}{p},
$$
for two density operators $\rho$ and $\sigma$, and a parameter $p$ such that $|p| \leq 1$ with $p \ne 0$.

\begin{remark}\label{remark3_2}
If we take two density operators $\rho$ and $\sigma$, $\Phi(\rho) \equiv Tr[\rho]$ and $f_p(t) = \frac{1-t^p}{p}$ for $-1 \leq p \leq 1$ with $p \ne 0$ in Proposition \ref{prop_3_2}, then we have the following inequalities by the similar way to Remark \ref{remark3_1},
\[\begin{array}{l}
\frac{{\left( {1 - p} \right)}}{2}\left( {{M^{p - 2}} - {m^{p - 2}}} \right)\left( {M + m - Mm} \right) + \frac{{\left( {1 - p} \right)}}{2}\left( {{m^{p - 2}} - {M^{p - 2}}Tr\left[ {\rho {{\left( {{\rho ^{ - 1/2}}\sigma {\rho ^{ - 1/2}}} \right)}^2}} \right]} \right)\\
 \le   Tr\left[ {{T_p}\left( {\rho \left| \sigma  \right.} \right)} \right]\\
 \le \frac{{\left( {1 - p} \right)}}{2}\left( {{m^{p - 2}} - {M^{p - 2}}} \right)\left( {M + m - Mm} \right) + \frac{{\left( {1 - p} \right)}}{2}\left( {{M^{p - 2}} - {m^{p - 2}}Tr\left[ {\rho {{\left( {{\rho ^{ - 1/2}}\sigma {\rho ^{ - 1/2}}} \right)}^2}} \right]} \right).
\end{array}\]

It is known the relation $D_p(\rho |\sigma) \leq -Tr[T_p(\rho | \sigma)]$ for $0 <  p \leq 1$ in \cite[Theorem 2.2]{FYK2004}, so that we obtain the following inequality for $0 <  p \leq 1$:
\[\begin{aligned}
  & {{D}_{p}}\left( \rho |\sigma  \right) \\
 & \le \frac{\left( 1-p \right)}{2}\left( {{m}^{p-2}}-{{M}^{p-2}} \right)\left( M+m-Mm \right)+\frac{\left( 1-p \right)}{2}\left( {{M}^{p-2}}Tr\left[ \rho {{\left( {{\rho }^{-\frac{1}{2}}}\sigma {{\rho }^{-\frac{1}{2}}} \right)}^{2}} \right] -{{m}^{p-2}}\right). \\
\end{aligned}\]
\end{remark}

\medskip

We have the following corollary which improves the non-negativity of quantum entropy and quantum Tsallis entropy.
\begin{corollary} \label{cor_3_2}
For a density operator $\rho >0$ and $-1\leq p \leq 1$ with $p\neq 0$, we have
$$
S_p(\rho) \geq \frac{(1-p)(M^{p+1}-m^{p+1})(1-M)(1-m)}{2m^{p+1}M^{p+1}} \geq 0.
$$
\end{corollary}
\begin{proof}
We take $\Phi(\rho) \equiv Tr[\rho] =1$ and $f_p(t) = \frac{t-t^{1-p}}{p}$ for $-1\leq p \leq 1$ with $p\neq 0$ on $(0,\infty)$ in \eqref{th1-2}. Since the spectrum $Sp(\rho)=(0,1)$, we take the interval $[m,M]$ such that $0<m \leq M \leq 1$. Then $\alpha = \frac{1-p}{M^{p+1}}$ and $\beta = \frac{1-p}{m^{p+1}}$. Since $Tr[\rho^2] \leq 1$, the inequality \eqref{th1-2} gives
\[\begin{aligned}
   \frac{Tr\left[ \rho -{{\rho }^{1-p}} \right]}{p}&\le \frac{\left( 1-p \right)}{2}\left( \frac{1}{{{m}^{p+1}}}-\frac{1}{{{M}^{p+1}}} \right)\left( M+m-Mm \right)+\frac{\left( 1-p \right)}{2}\left( \frac{1}{{{M}^{p+1}}}Tr\left[ {{\rho }^{2}} \right]-\frac{1}{{{m}^{p+1}}} \right) \\
 & \le \frac{\left( 1-p \right)\left( {{M}^{p+1}}-{{m}^{p+1}} \right)\left( M-1 \right)\left( 1-m \right)}{2{{m}^{p+1}}{{M}^{p+1}}}\le 0,
\end{aligned}\]
which implies the desired inequality.
\end{proof}
\begin{remark}
Taking the limit $p \to 0$ in Corollary \ref{cor_3_2}, we have
$$
S(\rho) \geq \frac{(M-m)(1-M)(1-m)}{2mM} \geq 0,
$$
which improves the non-negativity of von Neumann entropy.
\end{remark}

\medskip

Now, we apply Corollary~\ref{cor1} on a power function $f(t) =t^r$, $t \in (0,\infty)$, $r\in \mathbb{R}$. If  $r \in (-\infty, 0) \cup (1,\infty)$, then $\gamma >0$. Moreover,  if $r \in [-1,0] \cup [1,2]$ then $f(t) =t^r$ is operator convex, but if $r \in [0,1]$ then $f(t) =-t^r$ is operator convex. Taking this into account, we obtain the following results.  Details of the proof are left to the interested reader.

\begin{corollary} \label{cor2}
Let $A$ be a self-adjoint operator with $Sp\left( A \right)\subseteq \left[ m,M \right]\subset I$ for some scalars $0<m<M$ and let $\Phi :\mathcal{B}\left( \mathcal{H} \right)\to \mathcal{B}\left( \mathcal{K} \right)$ be a unital  positive linear map.
\begin{enumerate}[(i)]
\item If $r \in (-\infty, -1) \cup (2,\infty)$, then
\begin{equation*}
\begin{aligned}
  & \frac{1}{K\left( m,M,r \right)}\Phi \left( {{A}^{r}} \right) \\
 & \le \frac{1}{K\left( m,M,r \right)}\left\{ \Phi \left( {{A}^{r}} \right)+\frac{\gamma }{2}\left[ \left( M+m \right)\Phi \left( A \right)-Mm-\Phi \left( {{A}^{2}} \right) \right] \right\} \\
 & \le \Phi {{\left( A \right)}^{r}} \\
 & \le K\left( m,M,r \right) \Phi \left( {{A}^{r}} \right)-\frac{\gamma }{2}\left[ \left( M+m \right)\Phi \left( A \right)-Mm-\Phi {{\left( A \right)}^{2}} \right] \\
 & \le K\left( m,M,r \right)\Phi \left( {{A}^{r}} \right), \\
\end{aligned}
\end{equation*}
\end{enumerate}
where
$$\gamma = r (r-1) \cdot \min \left\{ m^{r-2}, M^{r-2} \right\},$$
and $K(m,M,r)$ is a generalized Kantorovich constant
$$
 K(m,M,r):= \frac{(m M^r - M m^r) }{(r-1)
(M-m)}\left(\frac{r-1}{r} \frac{M^r - m^r}{m M^r - M m^r}
\right)^r, \quad  r \in {\mathbb{R}}.
$$
Of course, $K(m,M,0)=K(m,M,1)=1$.
\begin{enumerate}[(ii)]
\item If $r \in [-1,0] \cup [1,2]$, then
\begin{equation}\label{cor2-2}
\begin{aligned}
  & \frac{1}{K\left( m,M,r \right)}\Phi \left( {{A}^{r}} \right) \\
 & \le \frac{1}{K\left( m,M,r \right)}\left\{ \Phi \left( {{A}^{r}} \right)+\frac{\gamma }{2}\left[ \left( M+m \right)\Phi \left( A \right)-Mm-\Phi \left( {{A}^{2}} \right) \right] \right\} \\
 & \le \Phi {{\left( A \right)}^{r}} \\
 & \le \Phi \left( {{A}^{r}} \right). \\
\end{aligned}
\end{equation}
The last inequality is due to \cite[Corollary 1.22]{FMPS2005}.
\end{enumerate}

\begin{enumerate}[(iii)]
\item If $r \in [0,1]$, then
\begin{equation}\label{cor2-3}
\begin{aligned}
  & \frac{1}{K\left( m,M,r \right)}\Phi \left( {{A}^{r}} \right) \\
 & \ge \frac{1}{K\left( m,M,r \right)}\left\{ \Phi \left( {{A}^{r}} \right)+\frac{r\left( 1-r \right)}{2{{M}^{2-r}}}\left[ \left( M+m \right)\Phi \left( A \right)-Mm-\Phi \left( {{A}^{2}} \right) \right] \right\} \\
 & \ge \Phi {{\left( A \right)}^{r}} \\
 & \ge \Phi \left( {{A}^{r}} \right). \\
\end{aligned}
\end{equation}
The last inequality is due to \cite[Corollary 1.22]{FMPS2005}.
\end{enumerate}
\end{corollary}

\begin{remark}
The inequalities \eqref{cor2-2} and \eqref{cor2-3} strengthens some well-known inequalities \cite[Theorem 1]{BD2001}. For example our result for $r=-1$:
\begin{equation} \label{ineq01_corollary}
\Phi(A^{-1} )\leq \frac{(M+m)^2}{4Mm} \Phi(A)^{-1} -\frac{(M+m) \Phi(A)-Mm-\Phi(A^2)}{M^3},
\end{equation}
 improves the  Kantorovich inequality
\begin{equation} \label{ineq01_remark}
\Phi(A^{-1} ) \leq \frac{(M+m)^2}{4Mm} \Phi(A)^{-1},
\end{equation}
given in \cite{MO} (see also \cite[Proposition 2.7.8]{Bhatia2007}).
\end{remark}

\medskip

Let us give an explicit simple example.
\begin{example}
We take the function  $f(t)=\frac{1}{t}$ on $0<m\leq t \leq M$.
For a positive operator $X$ on a Hilbert space $\mathcal{H}$, we also take $\Phi \left( X \right) \equiv\frac{1}{\dim \mathcal{H}} Tr\left[ X \right]$, and
\[A = \left( {\begin{array}{*{20}{c}}
3&{ - 2}\\
{ - 2}&7
\end{array}} \right).\]
The eigenvalues of $A$ are $5 \pm 2\sqrt{2}$ so that we take $m=2$ and $M=8$.
Then we have
$$
\frac{(M+m)^2}{4Mm} \left(\frac{Tr[A]}{2}\right)^{-1} -\frac{Tr[A^{-1}]}{2} =\frac{5}{272},
$$
and
$$
\frac{(M+m)^2}{4Mm} \left(\frac{Tr[A]}{2}\right)^{-1} -\frac{(M+m) \left(\frac{Tr[A]}{2}\right)-Mm-\left(\frac{Tr[A^2]}{2}\right)}{M^3} -\frac{Tr[A^{-1}]}{2} =\frac{143}{8704}.
$$
Since $\frac{5}{272}-\frac{143}{8704}=\frac{1}{512}$, this example shows our inequality (\ref{ineq01_corollary}) is better than the inequality (\ref{ineq01_remark}).
\end{example}

\section*{Acknowledgement}
The authors thank anonymous referees for giving valuable comments and suggestions to improve our manuscript.
The author (S.F.) was partially supported by JSPS KAKENHI Grant Number
16K05257.



\begin{thebibliography}{99}

\bibitem{BD2001}
R. Bhatia, C. Davis, {\it A key inequality for functions of matrices}, Linear Algebra Appl., {\bf323}(1-3) (2001), 1--5.

\bibitem{Bhatia2007}
R. Bhatia, {\it Positive definite matrices}, Princeton University Press, 2007.

\bibitem{6}
M.D. Choi, {\it A Schwarz inequality for positive linear maps on ${{C}^{*}}$-algebras}, Illinois J. Math., {\bf18} (1974), 565--574.

\bibitem{5}
Ch. Davis, {\it A Schwarz inequality for convex operator functions}, Proc. Amer. Math. Soc., {\bf8} (1957), 42--44.

\bibitem{dragomir}
S.S. Dragomir, {\it On New Refinements and Reverses of Young's Operator Inequality}, arXiv:1510.01314, (2015).

\bibitem{Eba}
A. Ebadian, I. Nikoufar, M. Eshagi Gordji, {\it Perspectives of matrix convex functions}, Proc. Natl. Acad. Sci. USA., {\bf108}(18) (2011), 7313--7314.

\bibitem{E}
E.G. Effros. {\it A matrix convexity approach to some celebrated quantum inequalities}, Proc. Natl. Acad. Sci. USA., {\bf106} (2009), 1006--1008.

\bibitem{FFS}
J.I. Fujii, M. Fujii, Y. Seo, {\it An extension of the Kubo-Ando theory: Solidarities}, Math. Japonica., {\bf35} (1990), 387--396.

\bibitem{FMPS2005} T.Furuta, J.Mi\'ci\'c Hot, J.Pe\v cari\' c and Y. Seo, {\it Mond-Pe\v cari\' c method in operator inequalities}, Element, Zagreb, 2005.

\bibitem{MO}
A.W. Marshall, I. Olkin, {\it Matrix versions of Cauchy and Kantorovich inequalities}, Aequationes Math., {\bf40} (1990), 89--93.

\bibitem{1}
J. Mi\'ci\'c, Z. Pavi\'c, J. Pe\v cari\'c, {\it Jensen's inequality for operators without operator convexity}, Linear Algebra Appl., {\bf434}(5) (2011), 1228--1237.

\bibitem{OP2004}
M. Ohya, D. Petz, {\it Quantum entropy and its use}, Springer, Second Edition 2004.

\bibitem{FYK2004}
S. Furuichi, K. Yanagi, K. Kuriyama, {\it Fundamental properties of Tsallis relative entropy}, J. Math. Phys., {\bf45} (2004), 4868--4876.

\bibitem{Furu2008}
S. Furuichi, {\it Matrix trace inequalities on the Tsallis entropies}, J. Inequal. Pure Appl. Math., {\bf9}, Art.1 (2008), 1--7.

\bibitem{FK1989}
J.I. Fujii, E. Kamei, {\it Relative operator entropy in non-commutative information theory}, Math. Japon., {\bf34} (1989), 341--348.

\bibitem{mpst2000}
J. Mi\' ci\' c, J. Pe\v cari\' c, Y. Seo, M. Tominaga, \textit{Inequalities for positive linear maps on Hermitian matrices}, Math. Inequal. Appl., \textbf{3} (2000), 559--591.

\bibitem{YKF2005}
K. Yanagi, K. Kuriyama, S. Furuichi, {\it Generalized Shannon inequalities based on Tsallis relative operator entropy}, Linear Algebra Appl., {\bf394} (2005), 109--118.
\end{thebibliography}
\end{document}